\documentclass{amsart}
\usepackage{latexsym,amsmath,amssymb}
\newtheorem{thm}{Theorem}[section]

\newtheorem{lem}[thm]{Lemma}

\theoremstyle{definition}
\theoremstyle{remark}

\numberwithin{equation}{section}

\begin{document}

\title[]
{Li-Yorke chaos for composition operators on Orlicz spaces }

\author{\sc\bf Y. Estaremi}
\address{\sc Y. Estaremi}
\email{y.estaremi@gu.ac.ir}

\address{Department of Mathematics, Faculty of Sciences, Golestan University, Gorgan, Iran.}
\thanks{}

\thanks{}



\date{}

\dedicatory{}

\commby{}

\begin{abstract}
In this paper we characterize Li-Yorke chaotic composition operators on Orlicz spaces. Indeed some necessary and sufficient conditions are provided for Li-Yorke chaotic composition operator $C_{\varphi}$ on the Orlicz space $L^{\Phi}(\mu)$. In some cases we have equivalent conditions for composition operators on Orlicz spaces to be Li-Yorke chaotic. The results of this paper extend similar results in $L^p$-spaces. Our results are essentially based on the results of \cite{bdp}.
\end{abstract}

\maketitle

\section{ \sc\bf Introduction and Preliminaries}
First of all, for the convenience of the reader, we recall some basic properties of Orlicz spaces that we use in the sequel. For more details on Orlicz spaces,  see \cite{kr,raor}.

A function $\Phi:\mathbb{R}\rightarrow [0,\infty]$ is called a \textit{Young's function}  if $\Phi$ is   convex, $\Phi(-x)=\Phi(x)$, $\Phi(0)=0$ and $\lim_{x\rightarrow \infty} \Phi(x)=+\infty$. With each \textit{Young's function} $\Phi$, one can associate another convex function $\Psi:\mathbb{R} \rightarrow [0, \infty]$ which is called complementary function to $\Phi$, and has the same properties of it, which is defined as follows:

$$\Psi(y)=\sup\{x\mid y\mid-\Phi(x): x\geq0\}, \ \ \ \ \ y\in \mathbb{R}.$$

It is immediately follows from the definition that $\Psi$ is also a Young's function since $\Psi(0)=0$, $\Psi(-y)=\Psi(y)$ and $\Psi(.)$ is a convex increasing function satisfying $\lim_{y\rightarrow \infty} \Psi(y)=+\infty$. Also, it is clear that the pair $(\Phi, \Psi)$ satisfies \textit{Young's inequality}:

$$xy\leq \Phi(x)+\Phi(y), \ \ \ \ \ \ \ \ x,y\in \mathbb{R}.$$

The generalized inverse of the \textit{Young's function} $\Phi$ is defined by
$$
\Phi^{-1}(y)=\inf \{ x\geq 0: \Phi(x)> y\} \quad (y\in [0,\infty)).
$$
Notice that if $x\geq0$, then $\Phi\big(\Phi^{-1}(x)\big)\leq x$,
and if $\Phi(x)<\infty$, we also have $x\leq\Phi^{-1}\big(\Phi(x)\big)$. There are equalities in either case when $\Phi$ is a \textit{Young's function} vanishing only at zero and taking only finite values.\\

An especially useful \textit{nice Young's function} $\Phi$, called an $N$-\textit{function}, is such that $\Phi(x)=0$ iff $x=0$, $\lim_{x\rightarrow 0}\frac{\Phi(x)}{x}=0$ and $\lim_{x\rightarrow \infty}\frac{\Phi(x)}{x}=+\infty$, while $\Phi(\mathbb{R})\subset \mathbb{R}^{+}$. Moreover, the complementary function to an $N$-function is again an $N$-function.\\

Let $\Phi$ be a \textit{Young's function}. Then we say $\Phi$ satisfies the
$\Delta_{2}$-condition, if $\Phi(2x)\leq
K\Phi(x) \; ( x\geq x_{0})$  for some constants
$K>0$ and $x_0>0$. Also, $\Phi$ satisfies the
$\Delta_{2}$-condition globally if $\Phi(2x)\leq
K\Phi(x) \; ( x\geq 0)$  for some
$K>0$.

The \textit{Young's function} $\Phi$ is said to satisfy the
$\Delta'$-condition (respectively, the $\nabla'$-condition), if there exist  $ d>0$
(respectively, $b>0$) and $x_0>0$ such that
$$
\Phi(xy)\leq d\,\Phi(x)\Phi(y) \quad (x,y\geq x_{0})
$$
$$
(\mbox{respectively, }  \Phi(bxy)\geq \Phi(x)\Phi(y) \quad ( x,y\geq x_{0})).
$$
If $x_{0}=0$, these conditions are said to hold
globally. Notice that if $\Phi\in \Delta'$, then  $\Phi\in
\Delta_{2}$.\\

For a given complete $\sigma$-finite measure space$(X,\mathcal{F},\mu)$, let $L^0(\mathcal{F})$  be the linear space of  equivalence classes of $\mathcal{F}$-measurable real-valued functions on $X$. For a given \textit{Young's function}  $\Phi$, the space
$$
L^{\Phi}(\mu)=\left\{f\in L^0(\mathcal{F}):\exists k>0,
\int_X\Phi(kf)d\mu<\infty\right\}
$$
is called an Orlicz space. Define the functional

$$N_{\Phi}(f)=\inf \{k>0:\int_{X}\Phi(\frac{f}{k})d\mu\leq 1\}.$$

The functional $N_{\Phi}(.)$ is a norm on $L^{\Phi}(\mu)$ and is called \textit{guage norm}(or Luxemburge norm). Also, $(L^{\Phi}(\mu), N_{\Phi}(.))$  is a normed linear space. If almost everywhere equal functions are identified, then $(L^{\Phi}(\mu), N_{\Phi}(.))$ is a Banach space, the basic measure space $(X,\mathcal{F},\mu)$ is unrestricted. Hence every element of $L^{\Phi}(\mu)$ is a class of measurable functions that are almost everywhere equal. Also, there is another norm on $L^{\Phi}(\mu)$, defined as follows:

$$\|f\|_{\Phi}=\sup\{\int_{X}\mid fg\mid d\mu: g\in B_{\Psi}\}=\sup\{\mid\int_{X}fg d\mu\mid: g\in B_{\Psi}\},$$
in which $B_{\Psi}=\{g\in L^{\Psi}(\mu): \int_{X}\Psi(\mid g\mid )d\mu\leq 1\}$.
  The norm $\|.\|_{\Phi}$ is called \textit{Orlicz norm}. For any $f\in L^{\Phi}(\mu)$, $\Phi$ being a Young function, we have
  $$N_{\Phi}(f)\leq \|f\|_{\Phi}\leq2N_{\Phi}(f).$$
  And also for every $F\in\mathcal{F}$ with $0<\mu(F)<\infty$, we have $N_{\Phi}(\chi_F)=\frac{1}{\Phi^{-1}(\frac{1}{\mu(F)})}$.\\

Throughout this paper, we denote by $(X,\mathcal{F},\mu)$, a measure space, that is, $X$ is a nonempty set, $\mathcal{F}$ is a sigma algebra on $X$ and $\mu$ is a positive measure on $\mathcal{F}$. Also, we assume that $\varphi:X\rightarrow X$ is a non-singular measurable transformation, that is, $\varphi^{-1}(F)\in \mathcal{F}$, for every $F\in \mathcal{F}$ and $\mu(\varphi^{-1}(F))=0$, if $\mu(F)=0$. Moreover, if there exists a positive constant $c$ for which\\
$$\mu(\varphi^{-1}(F))\leq c\mu(F), \ \ \ \ \text{for every} F\in \mathcal{F},$$
then the linear operator
$$C_{\varphi}:L^{\Phi}(\mu)\rightarrow L^{\Phi}(\mu), \ \  \ \ \ \ C_{\varphi}(f)=f\circ\varphi,$$
is well-defined and continuous on the Orlicz space $L^{\Phi}(\mu)$ and is called composition operator. For more details on composition operators on Orlicz spaces one can refer to \cite{chkm}.

In this paper we shall characterize Li-Yorke chaotic composition operators on Orlicz spaces. Indeed we provide some equivalent conditions for composition operators to be Li-Yorke chaotic.\\
Now we recall that a continuous map $f:\mathcal{M}\rightarrow \mathcal{M}$ on the metric space $(\mathcal{M},d)$ is called Li-Yorke chaotic if there exists an uncountable set $S\subseteq \mathcal{M}$ such that each pair $(x,y)\in S\times S$, such that $x\neq y$, is  a Li-Yorke pair for $f$, that is\\
$$\lim_{n\rightarrow \infty} \inf d(f^n(x),f^n(y))=0, \ \ \ \ \text{and} \ \ \ \lim_{n\rightarrow \infty} \inf d(f^n(x),f^n(y))>0.$$
The set $S$ is called scrambled set for $f$. If the scrambled set $S$ is dense (respectively, residual) in $\mathcal{M}$, then $f$ is said to be densely (respectively, generically) Li-Yorke chaotic. Characterizations for dense Li-Yorke chaos and generic Li-Yorke chaos were investigated in \cite{bbmp1}. Li-Yorke chaotic linear operators are studied in \cite{bbmp, bbmp1}. Specially, recently, Li-Yorke chaotic composition operators on $L^p$-spaces investigated in \cite{bdp}. Also, in \cite{bdp} the authors provided a necessary and sufficient condition on $\varphi$ for $C_{\varphi}$ to be topologically transitive or topologically mixing on $L^p(\mu)$. The reader can refer to \cite{bbmp2,bm,bm1,bcdmp,bmpp,gmm,m} and references therein, for a wide view of the linear dynamics.

\section{\sc\bf Li-Yorke chaotic composition operators}
In this section we recall some results about Li-Yorke chaotic continuous linear operators on Banach spaces \cite{bbmp,bbmp1}.\\
\begin{thm}\label{t1.1}
Let $T$ be a continuous linear operator on the Banach space $X$. Then the followings are equivalent:\\
\begin{itemize}
  \item $T$ is Li-Yorke chaotic;

  \item $T$ has a semi-irregular vector in $X$, in other words there exists $x\in X$ such that
  $$\lim_{n\rightarrow \infty}\inf \|T^nx\|=0 \ \ \text{and} \ \ \lim_{n\rightarrow \infty}\sup\|T^nx\|>0;$$

  \item $T$ has an irregular vector in $X$, in other words, there is a vector $x\in X$ such that\\
  $$\lim_{n\rightarrow \infty}\inf\|T^nx\|=0 \ \ \text{and} \ \ \lim_{n\rightarrow \infty}\sup\|T^nx\|=\infty.$$
\end{itemize}
\end{thm}
The following theorem gives necessary and sufficient conditions for composition operators on the Orlicz space $L^{\Phi}(\mu)$ to be Li-Yorke chaotic.

\begin{thm}
Let $C_{\varphi}$ be a continuous composition operator on the Orlicz space $L^{\Phi}(\mu)$. The $C_{\varphi}$ on the Orlicz space $L^{\Phi}(\mu)$ is Li-Yorke chaotic if and only if  there exist a non-empty countable collection of measurable sets $\{F_i\}_{i\in I}$ , with $0<\mu(F_i)<\infty$ and an increasing sequence $\{\beta_j\}_{j\in \mathbb{N}}\subseteq \mathbb{N}$ such that:\\
(I) $$\lim_{j\rightarrow \infty}\Phi^{-1}(\frac{1}{\mu(\varphi^{-\beta_j}(F_i))})=\infty, \ \ \ \text{for all} \ \ \ i\in I,$$\\
(II) $$\sup\{\frac{\Phi^{-1}(\frac{1}{\mu(F_i)})}{\Phi^{-1}(\frac{1}{\mu(\varphi^{-n}(F_i))})}\: i\in I, \ \ n\in \mathbb{N}\}=\infty.$$

\end{thm}
\begin{proof}
Let the composition operator $C_{\varphi}:L^{\Phi}(\mu):\rightarrow L^{\Phi}(\mu)$ be Li-Yorke chaotic. Then by Theorem \ref{t1.1} $C_{\varphi}$ has an irregular vector which we call it $f\in L^{\Phi}(\mu)$. We put
$$F_i=\{x\in X: 2^{i-1}\leq |f(x)|<2^i\}, \ \ \ \ \ i\in \mathbb{Z}.$$
It is clear that $F_i\in \mathcal{F}$, for all $i\in \mathbb{Z}$, and $X=\cup_{i\in I}F_i$, in which $I=\{i\in \mathbb{Z}:\mu(F_i)>0\}$. Since $f\in L^{\Phi}(\mu)$, then there exists $\alpha>0$ such that $\int_X\Phi(\alpha|f(x)|)d\mu<\infty$. This implies that for each $i\in I$ we have

 $$\Phi(2^{i-1}\alpha)\mu(F_i)\leq \int_X\Phi(\alpha|f(x)|)d\mu<\infty$$
 and so $\mu(F_i)<\infty$. By definition of irregular vectors we have $\lim_{n \rightarrow \infty}\inf N_{\Phi}(C^n_{\varphi}f)=0$. Hence we can find an increasing sequence $\{\beta_j\}_{j\in \mathbb{N}}$ of positive integers such that $\lim_{j\rightarrow \infty}N_{\Phi}(C_{\varphi}^{\beta_j}f)=0$. Now we have
 \begin{align*}
 \int_X\Phi(\frac{2^{i-1}\chi_{\varphi^{-\beta_j}(F_i)}}{N_{\Phi}(C^{\beta_j}_{\varphi}f)})d\mu&=\int_{\varphi^{-\beta_j}(F_i)}\Phi(\frac{2^{i-1}\chi_{F_i}\circ \varphi^{\beta_j}}{N_{\Phi}(C^{\beta_j}_{\varphi}f)})d\mu\\
 &\leq \int_{\varphi^{-\beta_j}(F_i)}\Phi(\frac{|f|\circ \varphi^{\beta_j}}{N_{\Phi}(C^{\beta_j}_{\varphi}f)})d\mu\\
 &\leq \int_{X}\Phi(\frac{|f|\circ \varphi^{\beta_j}}{N_{\Phi}(C^{\beta_j}_{\varphi}f)})d\mu\\
 &\leq 1.
 \end{align*}
This implies that $N_{\Phi}(2^{i-1}\chi_{\varphi^{-\beta_j}(F_i)})\leq N_{\Phi}(C^{\beta_j}_{\varphi}f)$ and so
 $$2^{i-1}\lim_{j\rightarrow \infty}N_{\Phi}(\chi_{\varphi^{-\beta_j}(F_i)})\leq \lim_{j\rightarrow \infty}N_{\Phi}(C^{\beta_j}_{\varphi}f)=0.$$
 Therefore for every $i\in I$,
 $$\lim_{j\rightarrow \infty}\frac{1}{\Phi^{-1}(\frac{1}{\mu(\varphi^{-\beta_j}(F_i))})}=0$$
 and consequently
 $$\lim_{j\rightarrow \infty}\Phi^{-1}(\frac{1}{\mu(\varphi^{-\beta_j}(F_i))})=\infty.$$

 Now consider that the condition (II) does not hold, it follows that there exists a positive constant $c>0$ such that
 \begin{equation}\label{e2.1}
 \frac{\Phi^{-1}(\frac{1}{\mu(F_i)})}{\Phi^{-1}(\frac{1}{\mu(\varphi^{-n}(F_i))})}\leq c, \ \ \ \ \ \text{for all} \ \ i\in I, \ \text{and} \ \ n\in \mathbb{N}.
 \end{equation}
 Since $\Phi\in \Delta_2$ (globally), then there exists $\alpha\geq 2$ such that $\Phi(2x)\leq \alpha\Phi(x)$, for every $x>0$. Also we can find $k>0$ such that $\frac{c}{2^k}\leq 1$. Therefore we have
 \begin{equation}\label{e2.2}
 \Phi(cx)=\Phi(\frac{c}{2^k}2^kx)\leq\frac{c\alpha^k}{2^k}\Phi(x).
 \end{equation}
 Hence by combining inequalities \ref{e2.1} and \ref{e2.2} we get that
 \begin{equation}\label{e2.3}
 \mu(\varphi^{-n}(F_i))\leq\frac{c\alpha^k}{2^k}\mu(F_i)\leq\alpha^k\mu(F_i), \ \ \ \ i\in I\ \ \text{and} \ \ \ \ n\in \mathbb{N}.
 \end{equation}
 Now we have
\begin{align*}
\int_X\Phi(\frac{f\circ \varphi^n}{2(\alpha^k+1)N_{\Phi}(f)})d\mu&=\Sigma_{i\in I}\int_{\varphi^{-n}(F_i)}\Phi(\frac{f\circ \varphi^n}{2(\alpha^k+1)N_{\Phi}(f)})d\mu\\
&\leq\Sigma_{i\in I}\int_{\varphi^{-n}(F_i)}\Phi(\frac{2^i}{2(\alpha^k+1)N_{\Phi}(f)})d\mu\\
&=\Sigma_{i\in I}\Phi(\frac{2^{i-1}}{(\alpha^k+1)\|f\|_{\Phi}})\mu(\varphi^{-n}(F_i))\\
&\leq\frac{\alpha^k}{\alpha^k+1} \Sigma_{i\in I}\Phi(\frac{2^{i-1}}{N_{\Phi}(f)})\mu(F_i)\\
&\leq\Sigma_{i\in I}\int_{F_i}\Phi(\frac{2^{i-1}}{N_{\Phi}(f)})d\mu\\
&\leq \Sigma_{i\in I}\int_{F_i}\Phi(\frac{|f|}{N_{\Phi}(f)})d\mu\\
&=\int_X\Phi(\frac{f}{N_{\Phi}(f)})d\mu\leq 1.\\
\end{align*}
This implies that $N_{\Phi}(C^n_{\varphi}f)\leq2(\alpha^k+1)N_{\Phi}(f)$. This is a contradiction, because $\lim_{n\rightarrow \infty}\sup N_{\Phi}(C^n_{\varphi}f)=\infty$.\\
Conversely, let conditions (I) and (II) hold. So there exists $\{F_i\}_{i\in I}\subseteq \mathcal{F}$, with $0<\mu(F_i)<\infty$  such that
$$\lim_{j\rightarrow \infty}\frac{1}{N_{\Phi}(C^{-\beta_j}_{\varphi}\chi_{F_i})}=\lim_{j\rightarrow \infty}\frac{1}{N_{\Phi}(\chi_{\varphi^{-\beta_j}(F_i)})}=\lim_{j\rightarrow \infty}\Phi^{-1}(\frac{1}{\mu(\varphi^{-\beta_j}(F_i))})=\infty, \ \ \ \text{for all} \ \ \ i\in I,$$\\
and
 $$\sup\{\frac{N_{\Phi}(C^{n}_{\varphi}\chi_{F_i})}{N_{\Phi}(\chi_{F_i})} i\in I, \ \ n\in \mathbb{N}\}=\sup\{\frac{\Phi^{-1}(\frac{1}{\mu(F_i)})}{\Phi^{-1}(\frac{1}{\mu(\varphi^{-n}(F_i))})}\: i\in I, \ \ n\in \mathbb{N}\}=\infty,$$
 for some sequence $\{\beta_j\}_{j\in \mathbb{N}}\subseteq \mathbb{N}$.\\
 Let $$\mathcal{L}=\overline{\text{linear span}\{\chi_{F_i}:i\in I\}}\subseteq L^{\Phi}(\mu),$$

 $$O_{C_{\varphi}}(f)=\{f,C_{\varphi}f,C^2_{\varphi}f,C^3_{\varphi}f,....\},$$

 $$\mathcal{L}_1=\{f\in \mathcal{L}: O_{C_{\varphi}}(f)\ \ \ \text{has a subsequence converging to zero}\},$$
 and
 $$\mathcal{L}_2=\{f\in \mathcal{L}: O_{C_{\varphi}}(f)\ \ \ \text{is unbounded}\}.$$
 From (I) we get that $\mathcal{L}_1$ is dense in $\mathcal{L}$. Hence by the Proposition 3 of \cite{bbmp1} we get that $\mathcal{L}_1$ is residual in $\mathcal{L}$. If we let $$f_i=\Phi^{-1}(\frac{1}{\mu(F_i)})\chi_{F_i}, \ \ \ \ \ i\in I,$$
 then for every $i\in I$, $f_i\in \mathcal{L}$ and
  $$N_{\Phi}(f_i)=1\ \ \ \ N_{\Phi}(C^n_{\varphi}f_i)=\frac{\Phi^{-1}(\frac{1}{\mu(F_i)})}{\Phi^{-1}(\frac{1}{\mu(\varphi^{-n}(F_i))})}.$$
  Thus by (II) we have
  $$\sup_{n\in \mathbb{N}}\sup_{f\in \mathcal{L}}N_{\Phi}(C^n_{\varphi}f)\geq\sup_{n\in \mathbb{N},i\in I}N_{\Phi}(C^n_{\varphi}f_i)=\sup_{n\in \mathbb{N}, i\in I}\frac{\Phi^{-1}(\frac{1}{\mu(F_i)})}{\Phi^{-1}(\frac{1}{\mu(\varphi^{-n}(F_i))})}=\infty.$$
   By applying Banach-Steinhaus Theorem we get $\mathcal{L}_2$ is residual in $\mathcal{L}$. Consequently, by Theorem \ref{t1.1} we obtain that $C_{\varphi}$ is Li-Yorke chaotic, because all elements of $\mathcal{L}_1\cap \mathcal{L}_2$ are irregular vectors for $C_{\varphi}$.
\end{proof}
Let $\varphi$ be injective and $F_i=\varphi^i(F)$, for $i\in\mathbb{Z}$. For $m,n$ with $n<m$ we have

$$F_n=\varphi^n(F)=\varphi^{n-m}(\varphi^m(F))=\varphi(F_m).$$
If we set
$$I=\{n\in\mathbb{Z}:0<\mu(\varphi^n(F))<\infty\},$$
then we have
\begin{align*}
\sup\{\frac{\Phi^{-1}(\frac{1}{\mu(\varphi^m(F))})}{\Phi^{-1}(\frac{1}{\mu(\varphi^n(F))})}:n,m\in I, \ \ n<m\}&=\sup\{\frac{\Phi^{-1}(\frac{1}{\mu(F_m)})}{\Phi^{-1}(\frac{1}{\mu(\varphi^{n-m}(F_m))})}:n,m\in I, \ \ n<m\}\\
&=\sup\{\frac{\Phi^{-1}(\frac{1}{\mu(F_m)})}{\Phi^{-1}(\frac{1}{\mu(\varphi^{n-m}(F_m))})}:n,m\in I, \ \ n<m \}.
\end{align*}
Therefore similar to the Corollary 1.2 of \cite{bdp} we get that if $\varphi$ is injective, then the composition operator $C_{\varphi}$ on the Orlicz space $L^{\Phi}(\mu)$ is Li-Yorke chaotic if there exists a measurable set $F$ with $0<\mu(F)<\infty$ such that\\
\begin{itemize}
  \item $\lim_{n\rightarrow \infty}\sup \Phi^{-1}(\frac{1}{\mu(\varphi^{-n})})=\infty$,
  \item $\sup\{\frac{\Phi^{-1}(\frac{1}{\mu(F_m)})}{\Phi^{-1}(\frac{1}{\mu(\varphi^{n-m}(F_m))})}:n,m\in I, \ \ n<m \}=\infty$.
\end{itemize}
%
%
Here we have a technical lemma for later use in this paper.
\begin{lem}\label{l2.2}
Let $\varphi:X\rightarrow X$ be an injective non-singular measurable transformation and $\Phi$ be a Young's function such that $\Phi\in \Delta_2$, globally. Then the following conditions are equivalent:\\
\begin{itemize}
  \item There exists $K>0$ such that for all $F\in \mathcal{F}$,
  $$\mu(\varphi^{-1}(F))\leq K\mu(F).$$
  \item There exists $L>0$ such that for all $F\in \mathcal{F}$,
  $$\Phi^{-1}(\frac{1}{\mu(F)})\leq\Phi^{-1}(\frac{L}{\mu(\varphi^{-1}(F))})\leq L\Phi^{-1}(\frac{1}{\mu(\varphi^{-1}(F))})$$
\end{itemize}

\end{lem}
\begin{proof}
It is an easy exercise.
\end{proof}
In the following interesting result we give some necessary and sufficient conditions  for a Li-Yorke chaotic composition operator C on Orlicz spaces . We prove that some of these assertions are equivalent under extra conditions.

\begin{thm} Consider the following items:\\
\begin{enumerate}
  \item The composition operator $C_{\varphi}:L^{\Phi}(\mu)\rightarrow L^{\Phi}(\mu)$ is Li-Yorke chaotic;
  \item There exists a non-zero measurable function $f\in L^{\Phi}(\mu)$ such that
  $$\lim_{n\rightarrow \infty}\inf N_{\Phi}(C^n_{\varphi}f)=0;$$
  \item There exists a measurable set $F\in \mathcal{F}$ with $0<\mu(F)<\infty$, such that
  $$\lim_{n\rightarrow \infty}\sup\Phi^{-1}(\frac{1}{\mu(\varphi^{-n}(F))})=\infty;$$
  \item There exists a measurable set $F$ with $0<\mu(F)<\infty$ such that
  $$\lim_{n\rightarrow \infty}\sup\Phi^{-1}(\frac{1}{\mu(\varphi^{n}(F))})=\infty;$$
  \item There exists a measurable set $F\in \mathcal{F}$ with $0<\mu(F)<\infty$, such that
  $$\lim_{n\rightarrow \infty}\sup\Phi^{-1}(\frac{1}{\mu(\varphi^{-n}(F))})=\infty \ \ \ \text{and} \ \ \lim_{n\rightarrow \infty}\sup\Phi^{-1}(\frac{1}{\mu(\varphi^{n}(F))})=\infty;$$
  \item There exists a measurable set $F\in \mathcal{F}$ with $0<\mu(F)<\infty$, such that
  $$\lim_{n\rightarrow \infty}\inf\Phi^{-1}(\frac{1}{\mu(\varphi^{-n}(F))})>0 \ \ \ \text{and} \ \ \lim_{n\rightarrow \infty}\sup\Phi^{-1}(\frac{1}{\mu(\varphi^{-n}(F))})=\infty;$$
  \item There exists a measurable set $F\in \mathcal{F}$ such that $\chi_{F}$ is a semi-irregular vector for the composition operator $C_{\varphi}:L^{\Phi}(\mu)\rightarrow L^{\Phi}(\mu)$.
\end{enumerate}
Then the following implications hold:\\
$$(6)\Leftrightarrow (7)\Rightarrow(1)\Rightarrow(2)\Leftrightarrow(3)\Leftarrow(5)\Rightarrow(4).$$
Moreover, if $\varphi$ is injective and $\Phi\in \Delta_2$, then we have the implication $(5)\Rightarrow (6)$ and if $\Phi\in\Delta_2$ and $\mu(X)<\infty$, then we have the implication $(4)\Rightarrow(5)$. Consequently, if $\Phi\in \Delta_2$, $\mu(X)<\infty$ and $\varphi$ is injective, then all conditions are equivalent.
\end{thm}
\begin{proof}
If $C_{\varphi}$ is Li-Yorke chaotic, then by Theorem \ref{t1.1} we have a semi-irregular function $f\in L^{\Phi}(\mu)$ for $C_{\varphi}$. Hence
$$\lim_{n\rightarrow \infty}\inf N_{\Phi}(C^n_{\varphi}f)=0$$
and
\begin{equation}\label{e2.4}
\lim_{n\rightarrow \infty}\sup N_{\Phi}(C^n_{\varphi}f)=\infty.
\end{equation}
The equation \ref{e2.4} implies that $f$ is non-zero. So we proved that the implication $(1)\Rightarrow(2)$ holds.\\
Let us prove the equivalence $(2)\Leftrightarrow(3)$. If $f$ satisfies the condition $(2)$, then there exists a positive number $\delta$ such that $$\mu(F=\{x\in X: |f(x)|>\delta\})>0$$ and also
$$\int_X\Phi(\frac{\delta\chi_{\varphi^{-k}(F)}}{N_{\Phi}(C^k_{\varphi}f)})d\mu\leq\int_X\Phi(\frac{|f|^k\circ \varphi}{N_{\Phi}(C^k_{\varphi}f)})d\mu\leq 1.$$
This means that
$$\delta\frac{1}{\Phi^{-1}(\frac{1}{\mu(\varphi^{-k}(F))})}=\delta N_{\Phi}(\chi_{F}\circ \varphi^k)\leq N_{\Phi}(C^k_{\varphi}f).$$

Therefore we get that
$$\delta\lim_{k\rightarrow \infty}\inf\frac{1}{\Phi^{-1}(\frac{1}{\mu(\varphi^{-k}(F))})}\leq \lim_{k\rightarrow \infty}\inf N_{\Phi}(C^k_{\varphi}f)=0$$
and consequently we have

$$\lim_{k\rightarrow \infty}\sup\Phi^{-1}(\frac{1}{\mu(\varphi^{-k}(F))})=\infty.$$
Hence the implication $(2)\Rightarrow(3)$ holds. For the converse, if there exists a measurable set $F\in \mathcal{F}$ with $0<\mu(F)<\infty$ such that
$$\lim_{k\rightarrow \infty}\sup\Phi^{-1}(\frac{1}{\mu(\varphi^{-k}(F))})=\infty,$$
then
\begin{align*}
\lim_{k\rightarrow \infty}\inf N_{\Phi}(\chi_{F}\circ \varphi^k)&=\lim_{k\rightarrow \infty}\inf\frac{1}{\Phi^{-1}(\frac{1}{\mu(\varphi^{-k}(F))})}\\
&=\frac{1}{\lim_{k\rightarrow \infty}\sup\Phi^{-1}(\frac{1}{\mu(\varphi^{-k}(F))})}=0.
\end{align*}
So the implication $(3)\Rightarrow(2)$ holds.\\
Let us prove the equivalence $(6)\Leftrightarrow(7)$. Suppose that $f=\chi_{F}$, in which $F\in \mathcal{F}$ with $0<\mu(F)<\infty$. Then we have
$$N_{\Phi}(C^k_{\varphi}f)=N_{\Phi}(\chi_{\varphi^{-k}(F)})=\frac{1}{\Phi^{-1}(\frac{1}{\mu(\varphi^{-k}(F))})}.$$
This equation shows that the properties $(6)$ and $(7)$ are equivalent. On the other hand, the implications $(7)\Rightarrow(1)$, $(5)\Rightarrow(3)$ and $(5)\Rightarrow(4)$ are clear.\\
Suppose that there exists a measurable set $F\in \mathcal{F}$ with $0<\mu(F)<\infty$, such that
  $$\lim_{n\rightarrow \infty}\sup\Phi^{-1}(\frac{1}{\mu(\varphi^{-n}(F))})=\infty \ \ \ \text{and} \ \ \lim_{n\rightarrow \infty}\sup\Phi^{-1}(\frac{1}{\mu(\varphi^{n}(F))})=\infty.$$
Since $\Phi\in\Delta_2$  replace and is by just(continuous and increasing), then we get that
$$\lim_{n\rightarrow \infty}\sup\frac{1}{\mu(\varphi^{-n}(F))}=\infty \ \ \ \text{and} \ \ \lim_{n\rightarrow \infty}\sup\frac{1}{\mu(\varphi^{n}(F))}=\infty,$$
and so we have
$$\lim_{n\rightarrow \infty}\inf\mu(\varphi^{-n}(F))=0 \ \ \ \text{and} \ \ \lim_{n\rightarrow \infty}\inf\mu(\varphi^{n}(F))=0.$$
Hence by Lemma \ref{l2.2} and Theorem 1.3 of \cite{bdp} there exists $F\in\mathcal{F}$ with $0<\mu(F)<\infty$ such that\\
$$\lim_{n\rightarrow\infty}\inf\mu(\varphi^{-n}(F))=0 \ \ \ \ \text{and} \ \ \ \ \lim_{n\rightarrow \infty}\sup\mu(\varphi^{-n}(F))>0.$$
Now by Lemma \ref{l2.2} we get the result i.e.,\\

$$\lim_{n\rightarrow \infty}\inf\Phi^{-1}(\frac{1}{\mu(\varphi^{-n}(F))})>0 \ \ \ \text{and} \ \ \lim_{n\rightarrow \infty}\sup\Phi^{-1}(\frac{1}{\mu(\varphi^{-n}(F))})=\infty.$$
Now let the condition (3) holds. Hence there exists a measurable set $F\in \mathcal{F}$ with $0<\mu(F)<\infty$, such that
  $$\lim_{n\rightarrow \infty}\sup\Phi^{-1}(\frac{1}{\mu(\varphi^{-n}(F))})=\infty.$$

Since $\Phi\in \Delta_2$, then by Lemma \ref{l2.2} we have\\
$$\lim_{n\rightarrow \infty}\sup\frac{1}{\mu(\varphi^{-n}(F))}=\infty$$
which is equivalent to
$$\lim_{n\rightarrow \infty}\inf\mu(\varphi^{-n}(F))=0.$$
We assume that $\mu(X)<\infty$, $\Phi\in \Delta_2$ and $\varphi$ is injective so by Theorem 1.3 of \cite{bdp} we get that there exists $F\in\mathcal{F}$ with $0<\mu(F)<\infty$ such that
$$\lim_{n\rightarrow \infty}\inf\mu(\varphi^{n}(F))=0.$$
Again by Lemma \ref{l2.2} we get that
$$\lim_{n\rightarrow \infty}\sup\Phi^{-1}(\frac{1}{\mu(\varphi^{n}(F))})=\infty.$$
so the condition (4) holds.\\
Let $\Phi\in\Delta_2$ and $\mu(X)<\infty$.
y using the same process of proving  the implication $(3)\Rightarrow(4)$, Lemma \ref{l2.2} and Theorem 1.3 of \cite{bdp} we get the  implication $(4)\Rightarrow (5)$.
\end{proof}
\textbf{Declarations}\\
     \textbf{Conflict of interest.} The authors have not disclosed any competing interests.
     \textbf{Acknowledgement.} My manuscript has no associate data.

\end{document}